\documentclass[11pt]{amsart}

\usepackage{hyperref}

\def\R{{\mathbb {R}}}
\def\N{{\mathbb {N}}}

\def\F{{\mathcal {F}}}

\def\diver{\operatorname {\mathrm{div}}}
\def\supp{\operatorname {\mathrm{supp}}}
\def\pint{\operatorname {--\!\!\!\!\!\int\!\!\!\!\!--}}

\def\ve{\varepsilon}

\newtheorem{teo}{Theorem}[section]
\newtheorem{lema}[teo]{Lemma}

\newtheorem{corol}[teo]{Corollary}

\theoremstyle{remark}
\newtheorem{remark}[teo]{Remark}

\theoremstyle{definition}
\newtheorem{defi}[teo]{Definition}

\numberwithin{equation}{section}

\title[Concentration-Compactness for Orlicz spaces]{The concentration-compactness principle for Orlicz spaces and applications}
\author[J. Fern\'andez Bonder and A. Silva]{Juli\'an Fern\'andez Bonder and Anal\'{\i}a Silva}

\address{Juli\'an Fern\'andez Bonder \hfill\break\indent
Instituto de C\'alculo (UBA - CONICET) and \hfill\break\indent
Departamento  de Matem\'atica, FCEyN, Universidad de Buenos Aires, \hfill\break\indent 
Pabell\'on I, Ciudad Universitaria (1428), Buenos Aires, Argentina.}

\email{{\tt jfbonder@dm.uba.ar}\hfill\break\indent {\it Web page:} {\tt http://mate.dm.uba.ar/$\sim$jfbonder}}

\address{Anal\'ia Silva \hfill\break\indent Departamento de Matem\'atica, FCFMyN, Universidad Nacional de San Luis and\hfill\break\indent 
Instituto de Matem\'atica Aplicada San Luis (IMASL), UNSL - CONICET.\hfill\break\indent 
Ejercito de los Andes 950, D5700HHW, San Luis, Argentina.}

\email{{\tt acsilva@unsl.edu.ar}\hfill\break\indent {\it Web page:} {\tt analiasilva.weebly.com}}



\begin{document}

\begin{abstract}
In this paper we extend the well-known concentration -- compactness principle of P.L. Lions to Orlicz spaces. As an application we show an existence result to some critical elliptic problem with nonstandard growth. 
\end{abstract}

\maketitle

\section{Introduction}
The study of Orlicz and Orlicz-Sobolev spaces is a subject that has a long history in analysis since the begginning of the 1930s starting with the work of Orlicz himself and the famous Polish school.

These spaces apear naturally in several applications in physics and engineering  when one need to deal with the so-called {\em nonstandard growth} differential equations. Some prototypical examples of such problems are equations of the form
\begin{equation}\label{eq.intro}
-\Delta_a u := -\diver\left(\frac{a(|\nabla u|)}{|\nabla u|}\nabla u\right) = f,
\end{equation}
posed in a domain $\Omega\subset\R^n$ and complemented with some boundary conditions.

The study of these problems is connected to Orlicz and Orlicz-Sobolev spaces since the natural space for solutions is $W^{1,A}(\Omega)$ where $A'(t)=a(t)$.

The regularity problem for \eqref{eq.intro} was analyzed in the classical work of Lieberman \cite{Lieberman} where it is shown, under adequate assumptions on $A$ and $f$, that bounded solutions are H\"older continuous.

The existence problem of \eqref{eq.intro} is related to the growth of the source term $f$ and therefore related to the integrability properties of functions in $W^{1,A}(\Omega)$ and for that purpose it is extremely relevant the study of the Sobolev immersions for these spaces. Namely, what is needed is to find all Young functions $B$ (see next section for precise definitions) such that
\begin{equation}\label{immersion.intro}
W^{1,A}(\Omega)\subset L^B(\Omega).
\end{equation}

As far as we know, the first article that treated this problem was \cite{Trudinger} and then it was refined in \cite{Cianchi} where the author finds the optimal Young function such that \eqref{immersion.intro} holds. In \cite{Cianchi} this optimal Young function is denoted by $A_n$ (it depends only on $A$ and $n$) and it is shown that \eqref{immersion.intro} holds if and only if $B\le A_n$ (in the sense of Young funcions). Moreover, if $B\ll A_n$ then the immersion \eqref{immersion.intro} is compact. This {\em critical} Young function $A_n$ has a precise formula given in \eqref{An}.

This compactness property of \eqref{immersion.intro} for $B\ll A_n$ is crucial for the existence problem of \eqref{eq.intro}. Namely, if $F'(t)=f(t)$ and $|F(t)|\le B(t)$ with $B\ll A_n$ then the standard variational methods (under adequate assumptions both on $A$ and $F$) can yield existence results for
$$
-\Delta_a u = f(u).
$$
However, for {\em critical-type} problems, where $F\sim A_n$ the existence problem becomes much more delicate.

In the classical setting (when the differential operator is the Laplace operator), this problem is very much related with the so-called Yamabe problem in differential geometry and the literature is so vast that is impossible to give an extensive review in this short introduction.

One extremely important tool to deal with such critical problems was developed by P.L. Lions in his famous article \cite{Lions}. P.L. Lions developed what is called as the {\em concentration-compactness principle} that consists in analyze the lack of compactness for bounded sequences in $W^{1,p}(\Omega)$. What Lions proved is that for bounded domains $\Omega$ the only possibility is the appearance of {\em concentration points}.

The concentration-compactness principle has been proved to be an extremely powerful tool and has been used by several authors in too many different problems and also it has been generalized to different settings. See \cite{Bonder-Saintier-Silva2, Bonder-Saintier-Silva, Bonder-Silva, Fu, Squassina} and references therein.

The main point of this article is therefore to generalized Lions' concentration-compactness principle to the context of Orlicz spaces. We want to remark that a first step in this direction was done in \cite{FIN}, where the concentration estimates were obtained by using global bounds on the Orlicz function $A$. Here we refine these results and get sharper estimates on the concentration of blowing-up sequences that depend only on the behavior of $A$ at infinity.

Then, as an application of the method, we give a proof of existence of solutions to
\begin{equation}\label{eq.a}
\begin{cases}
-\Delta_a u = a_n(u)+ \lambda f(u) & \text{in }\Omega\\
u=0 & \text{on }\partial\Omega,
\end{cases}
\end{equation}
where $f$ is a subcritical nonlinearity in the sense that $|F(t)|\le B(t)$ with $B\ll A_n$ (here $A_n'=a_n$). We want to remark that a very similar application was already analyze in \cite{FIN}.

\subsection*{Organization of the paper}
In section 2, we give a review of Young functions and Orlicz and Orlicz-Sobolev spaces that are needed in the course of the arguments. There are no new results there so any expert in the subject can safely skip this section.

In section 3, we prove some preliminary technical lemmas and in section 4 we prove our main result (Theorem \ref{ccp}) where we extend Lions' concentration-compactness principle to the Orlicz setting.

Finally, in section 5 we apply Theorem \ref{ccp} to obtain some existence result for \eqref{eq.a}.

\section{Young functions and Orlicz and Orlicz-Sobolev spaces: a review}

This section is devoted to give a very short overview of some known results on Young functions and Orlicz and Orlicz-Sobolev spaces that will be needed in the rest of the paper. There are no new results in this section so if the reader is familiar with the topic, he or she can skip the section and go directly to section 3. An excellent source for these topics is the classical book \cite{Krasnoselski}.

\subsection{Young functions}
Let us begin with the definition of a Young function.

\begin{defi}
A function $A\colon \R_+\to \R_+$ is said to be a Young function if it has the form
$$
A(t)=\int_0^t a(\tau)\,d\tau, \qquad t\geq 0,
$$
where $a: [0,\infty) \to [0,\infty)$ has the following properties:
\begin{itemize}
\item[(i)] $a(0)=0$,
\item[(ii)] $a(s)>0$ for $s>0$,
\item[(iii)] $a$ is right continuous at any point $s\geq 0$,
\item[(iv)] $a$ is nondecreasing on $(0,\infty)$.
\end{itemize}
\end{defi}

Associated to any Young function $A$ one define its {\em complementary function} (or Legendre conjugate) of $A$.
\begin{defi}
Let $A$ be a Young function, we define its complementary function $\tilde A$ as
$$
\tilde{A}(s): = \sup_{t\geq0}\{st-A(t)\}
$$
\end{defi}

Observe that, by definition, $\tilde A$ is the optimal function in Young's inequality
$$
st\le A(t) + \tilde A(s).
$$
It is a known fact that $\tilde A$ is also a Young function. Moreover, $\tilde A$ is given by
$$
\tilde A(s) = \int_0^s \tilde a(\tau)\, d\tau,
$$
where $\tilde a$ is the generalized inverse of $a$.

\bigskip

We need the notion of comparison between Young functions.
\begin{defi}
Given two Young functions $A$ and $B$, we say that $A\le B$ if there exists a constant $c>0$ and $t_0>0$ such that $A(t)\le B(ct)$, for every  $t\ge t_0$.

Whenever $A\le B$ and $B\le A$ we say that $A$ and $B$ are equivalent Young functions and this fact will be denoted by $A\sim B$.

Finally, we say that $B$ is essentially larger than $A$, denoted by $A\ll B$, if for any $c>0$,
$$
\lim_{t\to\infty} \frac{A(ct)}{B(t)}=0.
$$
\end{defi}

A very important and useful property is the so-called $\Delta_2-$condition. We recall this concept in the next definition.

\begin{defi}
We say that a Young function $A$ satisfies the $\Delta_2-$condition if
$$
A(2t)\leq C A(t)
$$
for all $t\geq 0$ for a fixed positive constant $C>1$.
\end{defi}

The next lemma is immediate and the proof is left to the reader.
\begin{lema}\label{delta2.refined}
Let $A$ be a Young function. The following are equivalent
\begin{enumerate}
\item $A$ verifies the $\Delta_2-$condition.

\item For every $\delta>0$, there exists $C_\delta>1$ such that
$$
A((1+\delta)t)\le C_\delta A(t).
$$

\item There exists $\delta_0>0$ and $C_0>1$ such that
$$
A((1+\delta_0)t)\le C_0 A(t).
$$
\end{enumerate}
\end{lema}

In \cite[Theorem 4.1]{Krasnoselski} it is shown that the $\Delta_2-$condition is equivalent to
\begin{equation}\label{p+}
\frac{ta(t)}{A(t)} \leq p^+  \quad \text{for } t>0,
\end{equation}
for some $p^+>1$.

Moreover, it can be checked that both $A$ and $\tilde A$ verify the $\Delta_2-$condition if and only if
\begin{equation}\label {delta}
 p^- \leq \frac{ta(t)}{A(t)} \leq p^+  \quad \text{for } t>0,
\end{equation}
where $1<p^-\le p^+<\infty$.

Let us now recall some simple inequalities for Young functions that will be helpful later on.
\begin{lema}[{\cite[Lemma 2.6]{Bonder-Salort}}]\label{Lemma2.3}
Let $A$ be a Young function satisfying \eqref{p+}. Then for every $\eta>0$ there exists $C_\eta>0$ such that
$$
A(s+t)\leq C_\eta A(s)+ (1+\eta)^{p^+}A(t)  \quad s,t>0.
$$
\end{lema}
\begin{lema}[{\cite[Lemma 2.1]{Bonder-Perez-Salort}}]\label{sacar}
Let $A$ be a Young function satisfying \eqref{delta}, $s,t>0$. Then
$$
\min\{s^{p^-},s^{p^+}\}A(t)\leq A(st)\leq\max\{s^{p^-},s^{p^+}\}A(t).
$$
\end{lema}

In order to understand the behavior of a Young function $A$ at infinity, it is very helpful to introduce the notion of the Matuszewska-Orlicz function and the Matuszewska-Orlicz index.
\begin{defi}\label{MO}
Given a Young function $A$, we define the associated Matuszewska-Orlicz function as
$$
M(t,A) := \limsup_{s\to\infty} \frac{A(st)}{A(s)}.
$$
When no confusion arises, we will simply denote $M(t)=M(t, A)$.

The Matuszewska-Orlicz index is then defined as
$$
p_\infty(A) := \lim_{t\to \infty} \frac{\ln M(t,A)}{\ln t} = \inf_{t>0} \frac{\ln M(t,A)}{\ln t}.
$$
Again, when no confusion arises, we will simply denote $p_\infty=p_\infty(A)$.
\end{defi}

The main feature that we use in this article is the fact that, if $A$ verifies the $\Delta_2-$condition, then the index $p_\infty$ is finite and for any $\ve>0$, there exists $t_0>0$ such that
\begin{equation}\label{recemos}
t^{p_\infty} \le M(t,A)\le t^{p_\infty + \ve}, \quad \text{for } t\ge t_0.
\end{equation}
See \cite{Arriagada} for this fact and more properties of this index.

\begin{remark}\label{MvsA}
It is also easy to check that if $A$ satisfy \eqref{delta}, then $p^-\le p_\infty\le p^+$. and that
$$
\min\{s^{p^+}, s^{p^-}\}M(t)\le M(st)\le \max\{s^{p^+}, s^{p^-}\} M(t).
$$
\end{remark}

We will need the following result regarding the function $M(\cdot, A)$.
\begin{lema}\label{M.orlicz}
If $A$ verifies \eqref{delta}, then $M$ is a Young function.
\end{lema}

\begin{proof}
According to \cite{Krasnoselski}, we need to check that $M(\cdot, A)$ is convex, even and verifies
$$
\lim_{t\to 0+} \frac{M(t, A)}{t} = 0 \quad \text{and}\quad \lim_{t\to\infty} \frac{M(t, A)}{t}=\infty.
$$
Note that if we call $A_s(t) := \frac{A(st)}{A(s)}$, then $A_s$ is convex and even. So $M(t, A)=\sup_{s>0}A_s(t)$ is also convex and even.

Also, $M(t, A)\ge A_1(t) = A(t)$ and so $\lim_{t\to\infty}  \frac{M(t, A)}{t} \ge \lim_{t\to\infty}  \frac{A(t)}{t} =\infty$.

Finally, from Lemma \ref{sacar}, $A(st)\le t^{p^-} A(s)$ for $t\in (0,1)$. Then it follows that $A_s(t)\le t^{p^-}$ for every $s>0$. Hence $M(t,A)\le t^{p^-}$ and so $\frac{M(t,A)}{t}\le t^{p^--1}\to 0$ as $t\to 0+$.
\end{proof}

\subsection{Orlicz and Orlicz-Sobolev spaces}
Given a Young function $A$ and an open set $\Omega\subset \R^n$ we consider the spaces $L^A(\Omega)$ and $W^{1,A}(\Omega)$ defined as follows:
\begin{align*}
L^A(\Omega)  :=\{ u \in L^1_\text{loc}(\Omega)\colon \Phi_{A,\Omega}(u)<\infty\},\
W^{1,A}(\Omega) :=\{ u\in W^{1,1}_\text{loc}(\Omega)\colon u, |\nabla u| \in L^{A}(\Omega)\},
\end{align*}
where
$$
\Phi_{A,\Omega}(u)= \int_\Omega A(|u|)\,dx.
$$
These spaces are endowed with the so-called \textit{Luxemburg norm} defined as
$$
\|u\|_{L^A(\Omega)}= \|u\|_{A; \Omega} = \|u\|_A :=\inf \left\{ \lambda>0 : \Phi_{A, \Omega}\left(\frac{u}{\lambda}\right)\leq 1 \right\}
$$
and
$$
\|u\|_{W^{1,A}(\Omega)} = \|u\|_{1,A: \Omega} = \|u\|_{1,A} := \|u\|_{L^A(\Omega)} + \| \nabla u \|_{L^A(\Omega)}.
$$
The spaces $L^A(\Omega)$ and
$W^{1,A}(\Omega)$ are separable, reflexive Banach spaces if and only if $A$ verifies \eqref{delta}. See \cite{Krasnoselski}.

From Lemma \ref{sacar} we immediately get
\begin{lema}\label{normayro}
Let $A$ be a Young function satisfying \eqref{delta}. Then the following inequlaity holds true for $u\in L^A(\Omega)$
$$
\min\{\|u\|_A^{p^-},\|u\|_A^{p^+}\}\leq\int_\Omega A(|u|)\, dx\leq \max\{\|u\|_A^{p^-},\|u\|_A^{p^+}\}.
$$
\end{lema}

In order for the Sobolev immersion theorem to hold, one need to impose some growth conditions on $A$. Following \cite{Cianchi},  we require $A$ to verify
\begin{align}
\label{A2} \int_K^\infty\left(\frac{t}{A(t)}\right)^\frac{1}{n-1}\,dt=\infty.\\
\label{A3} \int_0^\delta\left(\frac{t}{A(t)}\right)^\frac{1}{n-1}\,dt<\infty,
\end{align}
for some constants $K,\delta>0$.

Given a Young function $A$ that satisfies \eqref{A2} and \eqref{A3} its Orlicz-Sobolev conjugate is defined as
\begin{equation}\label{An}
A_n(t)=A\circ H^{-1} (t),
\end{equation}
where
\begin{equation}\label{H}
H(t)=\left(\int_0^t\left(\frac{\tau}{A(\tau)}\right)^\frac{1}{n-1}\,d\tau\right)^\frac{n-1}{n}.
\end{equation}

The following fundamental Orlicz-Sobolev embedding Theorem can be found in \cite{Cianchi}
\begin{teo}
Let $A$ be a Young function satisfying \eqref{A2} and \eqref{A3} and let $A_n$ be defined in \eqref{An}. Then the embedding $W_0^{1,A}(\Omega)\hookrightarrow L^{A_n}(\Omega)$ is continuous. Moreover, the Young function $A_n$ is optimal in the class of Orlicz spaces.

Finally, given $B$ any Young funcion, the embedding $W_0^{1,A}(\Omega)\hookrightarrow L^{B}(\Omega)$ is compact if and only if  $B\ll A_n$.
\end{teo}

From now on, we will denote by $S_A$ the optimal constant in the embedding $W^{1,A}_0(\Omega)\subset L^{A_n}(\Omega)$. That is
\begin{equation}\label{SA.def}
S_A := \inf_{\phi\in C^\infty_c(\Omega)} \frac{\|\nabla \phi\|_A}{\|\phi\|_{A_n}}.
\end{equation}

\begin{remark}
It is easy to see that $A\ll A_n$ and hence $W^{1,A}_0(\Omega)\subset L^A(\Omega)$ is compact.
\end{remark}

\section{Preliminary Lemmas}

In this section we prove some technical lemmas that will be helpful in the sequel. Namely, we need to show that $A_n$ defined in \eqref{An} verifies the $\Delta_2-$condition whenever $A$ does. Finally we need a version of the celebrated Brezis-Lieb Lemma to the Orlicz setting.

We begin with some preliminary estimates.
\begin{lema}\label{est.H}
Let $H(t)$ be as in definition \eqref{H}, then the following
inequality holds.
$$
C_1t^\frac{n}{n-p^-}\leq H^{-1}(t)\leq C_2t^\frac{n}{n-p^+},
$$
for $t>1$.
\end{lema}

\begin{proof}
From the definition of $H$, \eqref{H}, we get, for $t>1$,
\begin{align*}
H(t)^\frac{n}{n-1}&=\int_0^1\left(\frac{\tau}{A(\tau)}\right)^\frac{1}{n-1}\,d\tau+\int_1^t\left(\frac{\tau}{A(\tau)}\right)^\frac{1}{n-1}\,d\tau\\
&= C_0+\int_1^t\left(\frac{\tau}{A(\tau)}\right)^\frac{1}{n-1}\,d\tau,
\end{align*}
Observe that $C_0$ depends on $A$ and $n$.

Now, using Lemma \ref{sacar}, for $t>1$ we obtain
$$
A(1)^{\frac{-1}{n-1}}\int_1^t\tau^\frac{1-p_+}{n-1}\,d\tau\le \int_1^t\left(\frac{\tau}{A(\tau)}\right)^\frac{1}{n-1}\,d\tau\le A(1)^{\frac{-1}{n-1}}\int_1^t\tau^\frac{1-p_-}{n-1}\,d\tau.
$$
Observe that, for $1<p<n$,
$$
\int_1^t\tau^\frac{1-p}{n-1}\,d\tau = \frac{n-1}{n-p}(t^\frac{n-p}{n-1}-1).
$$
Then
$$
C_1t^\frac{n-p^+}{n}\leq H(t)\leq C_2t^\frac{n-p^-}{n}.
$$
From this last estimate we obtain the desired result.
\end{proof}

\begin{remark}\label{pn}
Combining Lemma \ref{sacar} with Lemma \ref{est.H} is easy to conclude that $A_n(t)$ verifies, for $t>1$,
$$
C_1t^{(p^-)_*}\leq A_n(t)\leq C_2t^{(p^+)_*},
$$
for some constants $C_1,C_2>0$ depending only on $A$ and $n$.

Recall that given an exponent $p\in (1,n)$ we denote by $p_*$ the Sobolev conjugate, $p_*=\frac{np}{n-p}$.
\end{remark}

Let us now check that the critical function $A_n$ inherits the $\Delta_2-$condition from $A$.
\begin{lema}
Let $A$ be a Young function satisfying \eqref{delta}, \eqref{A2} and \eqref{A3} and let $A_n$ be the Young function defined in \eqref{An}. Then, if $p^+<n$, $A_n$ verifies $\Delta_2$-condition.
\end{lema}

\begin{proof}
By definition of $H$ and using that $A(t)$ verifies the $\Delta_2$-condition , we obtain
\begin{align*}
H(2t)^\frac{n}{n-1}&=\int_0^{2t}\left(\frac{\tau}{A(\tau)}\right)^\frac{1}{n-1}\,d\tau  = 2^\frac{n}{n-1}\int_0^{t}\left(\frac{\tau}{A(2\tau)}\right)^\frac{1}{n-1}\,d\tau\\
&\geq 2^\frac{n}{n-1}\int_0^{t}\left(\frac{\tau}{2^{p^+} A(\tau)}\right)^\frac{1}{n-1}\,d\tau =2^{\frac{n-p^+}{n-1}} H(t)^\frac{n}{n-1}.
\end{align*}
Then
\begin{equation}\label{cotaH}
H(2t)\geq 2^{1-\frac{p^+}{n}} H(t).
\end{equation}
From \eqref{cotaH} we easily get that
$$
 2H^{-1}(s)\geq H^{-1}\left(2^{1-\frac{p^+}{n}} s\right).
$$
Observe that since $p^+<n$, $2^{1-\frac{p^+}{n}}= 1+\delta_0$ for some $\delta_0>0$.

Now, we are in position to prove that $A_n$ verifies $\Delta_2$-condition. In fact,
 $$
 A_n((1+\delta_0) t)=A(H^{-1}((1+\delta_0) t))\leq A\left(2H^{-1}\left(t\right)\right)\leq
 CA\left(H^{-1}\left(t\right)\right)= CA_n\left(t\right),
 $$
and the proof follows by Lemma \ref{delta2.refined}.
\end{proof}

To finish this section, we prove the Brezis-Lieb lemma in the Orlicz setting.
\begin{lema}[Brezis-Lieb Lemma]\label{Brezis-Lieb}\index{Brezis-Lieb lemma}
Let $B$ be a Young function, $f_n\to f$ a.e and $f_n\rightharpoonup f$ in $L^{B}(\Omega)$ then, for every $\phi\in L^\infty(\Omega)$ it follows that
$$
\lim_{n\to\infty} \left(\int_\Omega B(|f_n|)\phi \, dx - \int_\Omega B(|f-f_n|) \phi \, dx\right) = \int_\Omega B(|f|) \phi \, dx.
$$
\end{lema}
\begin{proof}
First, by Lemma \ref{Lemma2.3} we know that given $\ve>0$, there
exists $C_\ve$ such that for every $a,b\in\R$, the following inequality holds
$$
|B(|a+b|)-B(|a|)|\leq \varepsilon B(|a|)+C_{\varepsilon} B(|b|).
$$
We define
$$
W_{\varepsilon,n}(x)=(|B(|f_n(x)|)-B(|f(x)-f_n(x)|)-B(|f(x)|)|-\varepsilon
B(|f_n(x)|))_+,
$$
and note that $W_{\varepsilon,n}(x)\to 0$ as $n\to\infty$ a.e. On
the other hand,
\begin{align*}
|B(|f_n(x)|)- B(|f(x)-f_n(x)|)- B(|f(x)|)|&\leq|B(|f_n(x)|)-B(|f(x)-f_n(x)|)|+B(|f(x)|)\\
&\leq\varepsilon B(|f_n(x)|)+C_\varepsilon B(|f(x)|)+ B(|f(x)|),
\end{align*}
i.e.
$$
|B(|f_n(x)|)-B(|f(x)-f_n(x)|)-B(|f(x)|)|-\varepsilon B(|f_n(x)|)\leq
(C_\varepsilon+1)B(|f(x)|),
$$
therefore
$$
0\leq W_{\varepsilon,n}(x)\leq(C_\varepsilon+1) B(|f(x)|).
$$
By the dominated convergence Theorem, we conclude that
$$
\lim_{n\to\infty} \int_{\Omega} W_{\varepsilon,n}(x)\phi(x) \, dx = 0.
$$
On the other hand,
$$
|B(|f_n(x)|)-B(|f(x)-f_n(x)|)- B(|f(x)|)|\leq
W_{\varepsilon,n}(x)+\varepsilon B(|f_n(x)|).
$$
Then, if we denote
$
I_n=\int_\Omega \left(B(|f_n(x)|)- B(|f(x)-f_n(x)|)-B(|f(x)|)\right)\phi(x) \, dx,
$
we get
\begin{align*}
|I_n| &\le \int_\Omega  W_{\varepsilon,n}(x)|\phi(x)|\, dx + \varepsilon
\int_\Omega B(|f_n|) |\phi(x)|\, dx\\
& \le \int_\Omega W_{\varepsilon,n}(x) |\phi(x)|\, dx + \varepsilon \sup_{n\in\N}\int_\Omega B(|f_n|) |\phi|\, dx \\
&\le \int_\Omega W_{\varepsilon,n}(x) |\phi(x)|\, dx + \varepsilon C \|\phi\|_\infty,
\end{align*}
for some constant $C>0$. Hence, we can conclude that $\limsup
I_n\leq \varepsilon C\|\phi\|_\infty$, for every $\varepsilon>0$.
\end{proof}

\section{ Proof of the Concentration Compactness Principle}

This is the principal section of the paper where we prove the concentration compactness principle in the context of Orlicz spaces.

In this section we assume that the Young function $A$ satisfies condition \eqref{delta}.

Given $A$ a young function satisfying \eqref{delta}, we denote by $A_\infty$ the following Young function associated with $A$:
\begin{equation}\label{A.infty}
A_\infty(t)=\max\{t^{p^+},t^{p^-}\}.
\end{equation}
Observe that $A\le A_\infty$ both in the sense of Young functions and also in the pointwise sense.

\begin{remark}
$A_\infty$ verifies the $\Delta_2$-condition.
\end{remark}

In the sequel it will be helpful a comparison between the Orlicz functions $A_\infty$ and $M_n$, where $M_n(t)=M(t,A_n)$ is the Matuszewska-Orlicz function associated to $A_n$. This is the content of the next lemma.
\begin{lema}\label{AinftyMn}
With the same assumptions and notations of the section, assume that $p^+ < (p^-)_*$. Then
$$A_\infty\ll M_n$$
\end{lema}

\begin{proof}
This result is an immediate consequence of Remarks \ref{MvsA} and \ref{pn}.
\end{proof}

This next theorem is our main result.
\begin{teo}\label{ccp}
Let $\{u_k\}_{k\in\N}\subset W^{1,A}(\Omega)$ be a sequence such that $u_k\rightharpoonup u$ weakly in $W^{1,A}(\Omega)$. Then there exists a countable set $I$, positive numbers $\{\mu_i\}_{i\in I}$ and $\{\nu_i\}_{i\in I}$ such that
\begin{align}
\label{nu}& A_n(|u_k|)dx \rightharpoonup \nu = A_n(|u|)\,dx + \sum_{i\in I} \nu_i \delta_{x_i} \quad \text{weakly-* in the sense of measures,}\\
\label{mu}& A(|\nabla u_k|)\, dx \rightharpoonup \mu \geq A(|\nabla u|)\, dx + \sum_{i\in I} \mu_i \delta_{x_i} \quad \text{weakly-* in the sense of measures,}\\
\label{relacion}&S_A \frac{1}{M_n^{-1}(\frac{1}{\nu_i})}\leq\frac{1}{A_\infty^{-1}(\frac{1}{\mu_i})}, \qquad \text{for every } i\in I,
\end{align}
where $S_A$ is defined in \eqref{SA.def} and $M_n(t)=M(t, A_n)$ is the Matuszewska-Orlicz function associated to $A_n$. 
\end{teo}

The reader should compare Theorem \ref{ccp} with \cite[Lemma 4.2 ]{FIN}, where the authors obtain a similar result but the estimates on the concentration points are less sharp than ours. 

The strategy of the proof is the same as in the original work of P.L. Lions \cite{Lions}. 

Assume first that $u=0$ and so, passing if necessary to a subsequence,  $u_k\to 0$ a.e. in $\Omega$. We will prove Theorem \ref{ccp} in this case and then with the help of Lemma \ref{Brezis-Lieb} we can easily extend it to the general case.

We divide the proof of Theorem \ref{ccp} into a series of lemmas.The first one is a reverse-H\"older type inequality between the measures $\nu$ and $\mu$.

\begin{lema}\label{lema.rh}
Let $A$ be a Young function satisfying \eqref{delta}, \eqref{A2}, \eqref{A3} and let $A_n$ be given by \eqref{An}. Then, for every $\phi\in C^\infty_c(\Omega)$ the following reverse H\"older inequality holds:
\begin{equation}\label{RH}
S_A\|\phi\|_{M_n, \nu} \leq  \|\phi\|_{A_\infty,\mu},
\end{equation}
where $A_\infty$ is given by \eqref{A.infty} and  $M_n(t) = M(t, A_n)$ is the Matuszewska-Orlicz function associated to $A_n$ given in Definition \ref{MO}.
\end{lema}

\begin{proof}
Let $\phi\in C^\infty_c(\Omega)$ and we apply Sobolev inequality to $\phi u_k$, to obtain
\begin{equation}\label{RH1}
S_A\|\phi u_k\|_{A_n}\leq\|\nabla(\phi u_k)\|_{A}.
\end{equation}
First, we can estimate the left hand side in the following way: given $\delta>0$ let $K>0$ be such that
$$
A_n(st)\ge A_n(t) (M_n(s) - \delta),\quad \text{ for } t\ge K.
$$
Then
\begin{align*}
\liminf_{k\to\infty}\int_\Omega A_n\left(\frac{|\phi u_k|}{\lambda}\right)\,dx&\geq \liminf_{k\to\infty}\int_{\{u_k\ge K\}} A_n\left(\frac{|\phi u_k|}{\lambda}\right)\frac{A_n(|u_k|)}{A_n(|u_k|)}\,dx\\
&\geq\liminf_{k\to\infty}\int_{\{u_k\ge K\}} \left(M_n\left(\frac{|\phi|}{\lambda}\right)-\delta\right)A_n(u_k)\,dx\\
&= \liminf_{k\to\infty}\left(\int_\Omega - \int_{\{u_k< K\}}\right) \left(M_n\left(\frac{|\phi|}{\lambda}\right)-\delta\right)A_n(u_k)\,dx\\
& = \liminf_{k\to\infty} I - II.
\end{align*}

From \eqref{nu} it follows that
$$
\lim_{k\to\infty} I = \int_{\Omega} \left(M_n\left(\frac{|\phi|}{\lambda}\right)-\delta\right)\,d\nu
$$
and from the dominated convergence Theorem, since $u_k\to 0$ a.e., it follows that $\lim_{k\to\infty} II = 0$.

From these computations, since $\delta>0$ is arbitrary, one immediately obtain that
$$
\liminf_{k\to\infty} \int_\Omega A_n\left(\frac{|\phi u_k|}{\lambda}\right)\,dx\ge \int_\Omega M_n\left(\frac{|\phi|}{\lambda}\right)\, d\nu.
$$
From this inequality it follows that
$$
\liminf_{k\to\infty} \|\phi u_k\|_{A_n}\ge \|\phi\|_{M_n, d\nu}.
$$

Now we deal with the right hand side of \eqref{RH1}. First, we observe that
$$
|\ \|\nabla(\phi u_k)\|_{A} - \|\phi\nabla u_k\|_{A} |\le \| u_k \nabla \phi\|_{A}.
$$
Then, we observe that the right side of the inequality converges to 0 since $u_k\to 0$ in $L^A(\Omega)$.
Hence we can replace the right hand side of \eqref{RH1} by $\|\phi\nabla u_k\|_{A}$.

Now, using Lemma \ref{sacar},
\begin{align*}
\limsup_{k\to\infty}\int_{\Omega} A\left(|\nabla u_k| \frac{\phi}{\lambda}\right)\,dx&\le \limsup_{k\to\infty}\int_{\Omega}\max\left\{\left(\frac{\phi}{\lambda}\right)^{p^+},\left(\frac{\phi}{\lambda}\right)^{p^-}\right\} A(|\nabla u_k|)\,dx\\
&= \limsup_{k\to\infty} \int_{\Omega} A_\infty\left(\frac{\phi}{\lambda}\right) A(|\nabla u_k|)\,dx\\
&= \int_{\Omega} A_\infty\left(\frac{\phi}{\lambda}\right)\,d\mu,
\end{align*}
so
$$
\limsup_{k\to\infty} \|\phi\nabla u_k\|_{A}\leq\|\phi\|_{A_\infty,\mu}
$$
This inequality completes the proof.
\end{proof}

The next lemma is an easy adaptation of the first part of \cite[Lemma I.2]{Lions}. We include the details for completeness.
\begin{lema}\label{Lema 3}
Let $\nu$ be a non-negative, bounded Borel measure and let $A, B$ be two Young functions such that $A\ll B$. If
\begin{equation}\label{rh-lema3}
\|\phi\|_{B, \nu} \leq C \|\phi\|_{A,\nu},
\end{equation}
for some constant $C>0$ and for every $\phi\in C^\infty_c(\Omega)$. Then there exists $\delta>0$ such that for all Borel sets $U \subset \overline{\Omega}$, either $\nu(U)=0$ or $\nu(U)\ge\delta$.
\end{lema}

\begin{proof}
It is easy to see that inequality \eqref{rh-lema3} still holds for characteristic functions of Borel sets. Then we may take $\phi=\chi_U$ with $\nu(U)\neq0$, so
$$
\int_{\Omega} B\left(\frac{\chi_U}{\lambda}\right)\, d\nu = \int_U B\left(\frac{1}{\lambda}\right)\, d\nu = B\left(\frac{1}{\lambda}\right)\nu(U).
$$
Then
\begin{equation}\label{norm.char}
\|\chi_U\|_{B, \nu} = \frac{1}{B^{-1}\left(\frac{1}{\nu(U)}\right)}.
\end{equation}
Analogously,
$$
\|\chi_U\|_{A, \nu} = \frac{1}{A^{-1}\left(\frac{1}{\nu(U)}\right)}.
$$
Therefore we obtain the following inequality $A^{-1}\left(\frac{1}{\nu(U)}\right)\le C B^{-1}\left(\frac{1}{\nu(U)}\right)$.

Assume by contradiction that there exist $U_k$ such that $\nu(U_k) = \varepsilon_k\to 0$, so
$$
A^{-1}\left(\frac{1}{\varepsilon_k}\right)\leq C B^{-1}\left(\frac{1}{\varepsilon_k}\right).
$$
We choose $t_k$ such that $B(t_k)=\frac{1}{\varepsilon_k}$. Then $t_k\to\infty$ and  $A^{-1}(B(t_k))\leq Ct_k$. By composition with $A$ we obtain that $B(t_k)\leq A(Ct_k)$, which contradicts the fact that $A\ll B$.
\end{proof}

The next lemma is exactly as the end of \cite[Lemma I.2]{Lions}.
\begin{lema}\label{Lema 2}
Let $\nu$ be a non-negative bounded Borel measure on $\overline{\Omega}$. Assume that there exists $\delta>0$ such that for every Borel set $U$ we have that, $\nu(U)=0$ or $\nu(U)\geq\delta$. Then, there exist a countable index set $I$, points $\{x_i\}_{i\in I}\subset \bar\Omega$ and scalars $\{\nu_i\}_{i\in I}\in (0,\infty)$ such that
$$
\nu = \sum_{i\in I} \nu_i\delta_{x_i}.
$$
\end{lema}

Now we need a lemma that plays a key role in the proof of Theorem \ref{ccp}.
\begin{lema}\label{Lema 1}
Under the same assumptions of Lemma \ref{lema.rh}, there exist a countable index set $I$, points $\{x_i\}_{i\in I}\subset\bar{\Omega}$ and scalars $\{\nu_i\}_{i\in I}\subset (0,\infty)$, such that
$$
\nu=\sum_{i\in I} \nu_i\delta_{x_i}.
$$
\end{lema}

\begin{proof}
By the reverse H\"older inequality \eqref{RH}, the measure
$\nu$ is absolutely continuous with respect to $\mu$. In fact, if we
choose $\phi=\chi_U$, by \eqref{norm.char},
$$
\|\chi_U\|_{A_\infty, \mu} = \begin{cases} 0& \text{if } \mu(U)=0\\
\frac{1}{A_\infty^{-1}\left(\frac{1}{\mu(U)}\right)}&\text{if } \mu(U)>0.
\end{cases}
$$
Also, 
$$
\|\chi_U\|_{M_n, \nu} = \begin{cases} 0& \text{if } \nu(U)=0\\
\frac{1}{M_n^{-1}\left(\frac{1}{\nu(U)}\right)}&\text{if } \nu(U)>0.
\end{cases}
$$
This facts together with \eqref{RH} clearly imply that $\nu\ll\mu$.

 As a consequence there exists $f\in L_\mu^1(\Omega)$, $f\geq0$, such that $\nu=\mu\lfloor f$. Also by \eqref{RH} we can conclude that
$f\in L^\infty_\mu(\Omega)$. In fact,
$$
\pint_U f\, d\mu = \frac{\nu(U)}{\mu(U)} \le \frac{1}{\mu(U) M_n\left(C A_\infty^{-1}\left(\frac{1}{\mu(U)}\right)\right)}.
$$
Observe that if we denote $t=A_\infty^{-1}\left(\frac{1}{\mu(U)}\right).$ We have the following equality
$$
\mu(U) M_n\left(C A_\infty^{-1}\left(\frac{1}{\mu(U)}\right)\right) = \frac{M_n(C t)}{A_\infty(t)},
$$ and the last term goes to $\infty$ as $t\to\infty$ by Lemma \ref{AinftyMn}.

In other words, the function $r \mapsto \frac{1}{r M_n\left(C A_\infty^{-1}\left(\frac{1}{r}\right)\right)}$ is bounded
in $[0,\mu(\Omega))$, then $f\in L^\infty_\mu(\Omega)$.

On the other hand the Lebesgue decomposition of $\mu$ with respect to $\nu$ gives us
$$
\mu=\nu\lfloor g + \sigma,
$$
where $g\in L^1_\nu(\Omega)$, $g\geq0$ and $\sigma$ is a bounded positive measure, singular with respect to $\nu$.

Let $\psi\in C^\infty_c(\Omega)$ and consider \eqref{RH} applied to the test functions of the form
$\varphi(g)\psi\chi_{\{g\leq k\}}$ where $\varphi(t)$ is to be determined, $\varphi(0)=0$.

We obtain
\begin{align*}
C \|\varphi(g)\psi\chi_{\{g\leq k\}}\|_{M_n,\nu}&\leq\|\varphi(g)\psi\chi_{\{g\leq k\}}\|_{A_\infty,\mu}\\
&\leq\|\varphi(g)\psi\chi_{\{g\leq k\}}\|_{A_\infty,gd\nu+d\sigma}.\\
\end{align*}
Since $\sigma\perp\nu$, $A_\infty$ is sub-multiplicative (i.e. $A_\infty(st)\le A_\infty(s) A_\infty(t)$) and $A_\infty(\chi_U)=\chi_U$, we have that
\begin{align*}
\int_\Omega A_\infty\left(\frac{\varphi(g)\psi\chi_{\{g\leq k\}}}{\lambda}\right)\, d\mu & = \int_\Omega A_\infty\left(\frac{\varphi(g)\psi\chi_{\{g\leq k\}}}{\lambda}\right)g\,d\nu + \int_\Omega A_\infty\left(\frac{\varphi(g)\psi\chi_{\{g\leq k\}}}{\lambda}\right)\,d\sigma\\
&=\int_\Omega A_\infty\left(\frac{\varphi(g)\psi\chi_{\{g\leq k\}}}{\lambda}\right)g\,d\nu\\
&\leq\int_\Omega
A_\infty\left(\frac{\psi}{\lambda}\right)A_\infty(\varphi(g)) g\chi_{\{g\leq k\}}\,d\nu.
\end{align*}

On the other hand, combining remarks \ref{pn} and \ref{MvsA}, we have
\begin{align*}
\int_\Omega M_n\left(\frac{\varphi(g) \psi \chi_{\{g\le k\}}}{\lambda}\right)\, d\nu\ge \int_\Omega M_n\left( \frac{\psi}{\lambda}\right) \min\{\varphi(g)^{p^+_*}, \varphi(g)^{p^-_*}\} \chi_{\{g\le k\}}\, d\nu 
\end{align*}

Hence, if we define
$$
\varphi(t)=\begin{cases}
t^{\frac{1}{p^+_* - p^-}} & \text{if } t<1\\
t^{\frac{1}{p^-_* - p^+}} & \text{if } t\ge 1
\end{cases}
$$
then
$$
\max\{\varphi(t)^{p^+}, \varphi(t)^{p^-}\} t = \min\{\varphi(t)^{p^+_*}, \varphi(t)^{p^-_*}\}
$$
and we get that
$$
h(x) := A_\infty(\varphi(g(x))) g(x) = \min\{\varphi(g(x))^{p^+_*}, \varphi(g(x))^{p^-_*}\}
$$
Hence if we denote $d\nu_k := h(x) \chi_{\{g\leq k\}} d\nu$ the following
reverse H\"older inequality holds
$$
C\|\psi\|_{M_n,\nu_k}\le \|\psi\|_{A_\infty,\nu_k}.
$$
Now, by Lemmas \ref{Lema 2} and \ref{Lema 3}, there exists $\{x_i^k\}_{i\in I^k}$ and $\nu_i^k>0$ such that $\nu_k=\sum_{i\in
I^k} \nu_i^k\delta_{x_i^k}$. On the other hand, $\nu_k\nearrow h(x)\nu$. Then, we have
$$
\nu=\sum_{i\in I} \nu_i \delta_{x_i}.
$$
This finishes the proof.
\end{proof}

Now we are in position to prove Theorem \ref{ccp}.

\begin{proof}[Proof of Theorem \ref{ccp}]
First we write $v_k=u_k-u$. Then, we can apply Lemmas \ref{lema.rh}--\ref{Lema 1} to conclude that
\begin{equation}\label{nu.1}
A_n(|v_k|)\, dx \rightharpoonup d\bar\nu = \sum_{i\in I}\nu_i \delta_{x_i},
\end{equation}
weakly star in the sense of measures.

Now, we use Lemma \ref{Brezis-Lieb} to obtain
$$
\lim_{k\to\infty}\left(\int_\Omega \phi A_n(|u_k|)-\int_\Omega\phi A_n(|v_k|) dx\right)=\int_\Omega \phi A_n(|u|) dx,
$$
for any $\phi\in C^\infty_c(\Omega)$, from where the representation
$$
A_n(|u_k|)\, dx \rightharpoonup d\nu=A_n(|u|)\, dx + d\bar\nu
$$
follows.

It remains to analyze the measure $\mu$ and to estimate the weights $\nu_i$ and $\mu_i$.

To this end, we consider again $v_k=u_k-u$ and denote by $\bar\mu$ the weak* limit of $A(|\nabla v_k|)\, dx$ as $k\to\infty$.

Let $\phi\in C^\infty_c(\R^n)$ be such that $0\leq\phi\leq1$, $\phi(0)=1$ and supp$(\phi)\subset B_1(0)$. Now, for each $i\in I$ and $\varepsilon>0$, we denote $\phi_{\varepsilon,i}(x):= \phi((x-x_i)/\varepsilon)$.

Now we apply \eqref{RH} to the measures $\bar\nu$ and $\bar\mu$ to obtain
$$
S_A\frac{1}{M_n^{-1}(\frac{1}{\nu_i})}\le C \|\phi_{\varepsilon,i}\|_{M_n,\bar\nu} \leq \|\phi_{\varepsilon,i}\|_{A_\infty,\bar\mu},
$$

On the one hand,
$$
1=\int_{B_\varepsilon(x_i)} A_\infty\left(\frac{|\phi_{\varepsilon, i}|}{ \|\phi_{\varepsilon, i}\|_{A_\infty, \bar\mu}}\right)\,d\bar\mu \le A_\infty\left(\frac{1}{ \|\phi_{\varepsilon, i}\|_{A_\infty, \bar\mu}}\right) \bar\mu(B_{\varepsilon}(x_i)),
$$
hence,
$$
\|\phi_{\varepsilon, i}\|_{A_\infty, \bar\mu} \le \frac{1}{A_\infty^{-1}(\frac{1}{\bar\mu(B_\ve(x_i))})}\to \frac{1}{A_\infty^{-1}(\frac{1}{\bar\mu_i})} \quad \text{as }\ve \to 0,
$$
where
$$
\bar \mu_i := \mu(\{x_i\}) = \lim_{\ve\to 0} \bar\mu(B_{\varepsilon}(x_i)).
$$
Therefore,
$$
\bar\mu\ge \sum_{i\in I} \bar\mu_i \delta_{x_i} \qquad \text{and}\qquad S_A \frac{1}{M_n^{-1}(\frac{1}{\nu_i})}\le \frac{1}{A_\infty^{-1}(\frac{1}{\bar\mu_i})} 
$$
On the other hand, using Lemma \ref{Lemma2.3}, we have that for any $\delta>0$ there exists a constant $C_\delta$ such that
$$
A(|\nabla v_k|)\le (1+\delta) A(|\nabla u_k|) + C_\delta A(|\nabla u|).
$$
This inequality implies, passing to the limit $k\to\infty$, that
$$
d\bar\mu\le (1+\delta) d\mu + C_\delta A(|\nabla u|) \, dx,
$$
from where it follows that
$$
\bar\mu_i \le (1+\delta) \mu_i,\quad \text{where } \mu_i := \mu(\{x_i\}).
$$
This shows that $\mu \ge \sum_{i\in I} \mu_i \delta_{x_i} =\tilde\mu$ and, since $\delta>0$ is arbitrary, we get
$$
S_A \frac{1}{M_n^{-1}(\frac{1}{\nu_i})}\le \frac{1}{A_\infty^{-1}(\frac{1}{\mu_i})}.
$$

To end the proof it remains to show that $d\mu\geq A(|\nabla u|)\, dx$ since $\tilde\mu$ is orthogonal to the Lebesgue measure.

Now, the fact that $u_k\rightharpoonup u$ weakly in $W_0^{1,A}(\Omega)$ implies that $\nabla u_k\rightharpoonup\nabla u$ weakly in $L^{A}(U)$ for all $U\subset\Omega$. Hence, since the modular is a convex and strongly continuous functional, by \cite[Corollary 3.9]{Brezis} it follows that it is weakly lower semicontinuous. Hence we obtain that $d\mu\geq A(|\nabla u|)\,dx$ as we wanted to show.

This finishes the proof.
\end{proof}

\section{Application}
In this section, we study the existence problem for the following elliptic equation
\begin{equation}\label{aplicacion}
\begin{cases}
-\Delta_{a} u= \frac{a_n(|u|)}{|u|}u+\lambda f(u) & \mbox{in }\Omega,\\
u=0&\mbox{ in }\partial\Omega,
\end{cases}
\end{equation}
where $\Delta_a u =\diver\left(\frac{a(|\nabla u|)\nabla u}{|\nabla u|}\right)$, $A'(t)=a(t)$, $A_n'(t)=a_n(t)$ and if $F(t)=\int_0^{|t|} f(s)\, ds$, then $|F(t)|\le B(t)$ for some Young function $B$ such that $B\ll A_n$. This application is similar to the one considered in \cite{FIN}.

In this case, the associated functional reads
\begin{equation}\label{Flambda}
\F_\lambda(u)=\int_\Omega A(|\nabla u|)-A_n(|u|)-\lambda F(u)\,dx.
\end{equation}

For this problem we can prove the following result
\begin{teo}\label{r>p}
Let $A$ be a Young function satisfying the hypotheses of Theorem \ref{ccp}.

Assume that $f(t)$ verifies the {\em Ambrosetti-Rabinowitz} condition, i.e. there exists $\gamma>1$ such that
$$
f(t)t\le \gamma F(t),\qquad \text{with } p^+<\gamma<p_n^-.
$$

Moreover, assume that
$$
p^-\leq\frac{ta(t)}{A(t)}\leq p^+,\quad r^-\leq\frac{tb(t)}{B(t)}\leq r^+, \quad p^+<r^-<p_n^- \quad \text{ and } \quad B\ll A_n,
$$
where $B'(t)=b(t)$.

Then, there exists $\lambda_0>0$, such that if $\lambda>\lambda_0$ problem \eqref{aplicacion} has at least one nontrivial solution in $W^{1,A}_0(\Omega)$.
\end{teo}

Our proof relies in the following theorem, its proof can be found in \cite{AE}.
\begin{teo}\label{MPL}
Let $E$ be a Banach space and $I \in C^1(E, \mathbb{R})$. Suppose that there exist a neighbourhood $U$ of $0$ in $E$ and a constant $\alpha$ satisfying the following conditions 
\begin{itemize}
\item[(i)] $I(u) \geq \alpha$ for all $u \in \partial U$,
\item[(ii)] $I(0)< \alpha$,
\item[(iii)] there exists $u_0 \notin U$ satisfying $I(u_0)< \alpha$.
\end{itemize}Let
$$
\Gamma = \left\lbrace \gamma \in C([0, 1], E): \gamma(0)=0,\,\gamma(1)=u_0\right\rbrace
$$
and
$$
c=\inf_{\gamma \in \Gamma}\max_{u \in \gamma([0, 1])}I(u) \geq \alpha.
$$
Then, there exists a sequence $u_k \in E$ such that
$$
I(u_k)\to c \quad \text{and }\quad I'(u_k)\to 0 \,\,\text{in }E',
$$
as $k \to \infty$.
\end{teo}

Next, we begin to show some important properties of the functional $\F_\lambda$ defined in \eqref{Flambda}.
\begin{lema}\label{I satisfies MPL}
The functional $\F_\lambda$ defined in \eqref{Flambda} is $C^1$ and satisfies the geometric conditions {\em (i)--(iii)} from Theorem \ref{MPL}.
\end{lema}

\begin{proof}
It is easy to see that $\F_\lambda$ is $C^1$. First, observe that $\F_\lambda(0)=0 $, and  for $\alpha >0$ it is clear that $\F_\lambda$ verifies (ii). 
 
To prove that $\F_\lambda$ satisfies (i), let $U=B(0, \rho)$, with $0<\rho<1 $ to be chosen. If $u \in \partial B(0, \rho)$, then
$$
\|\nabla u\|_A=\rho.
$$
Hence,
\begin{equation*}
\begin{split}
\F_\lambda(u)&=\int_\Omega A(|\nabla u|)-A_n(|u|)-\lambda F(|u|)\,dx\\ &\geq \|\nabla u\|_A^{p^+}-\max\{\| u\|_{A_n}^{p_n^-},\| u\|_{A_n}^{p_n^+}\}-\lambda \max\{\|u\|_{B}^{r^-},\|u\|_{B}^{r^+}\} \\ & \geq \|\nabla u\|_A^{p^+}-\max\{\|\nabla u\|_{A}^{p_n^-},\|\nabla u\|_{A}^{p_n^+}\}-\lambda C\max\{\|\nabla u\|_{A}^{r^-},\|\nabla u\|_{A}^{r^+}\} \\&\geq \|\nabla u\|_A^{p^+}-C\|\nabla u\|_A^{p_n^-}-\lambda C\|\nabla u\|_A^{r^-}\\&\geq\rho^{p^+}-C\rho^{p_n^-}-\lambda C\rho^{r^-} \geq \alpha,
\end{split}
\end{equation*}
for some $\alpha >0$, choosing $\rho$ small enough and recalling $p^+<r^-<p_n^-$.

Finally, we prove that $\F_\lambda$ satisfies (iii). Let us fix $u\in W^{1,A}_0(\Omega)$ such that $\|\nabla u\|_{A}=1$. For $t >0$,  we consider $tu$. Then,
\begin{align*}
\F_\lambda(tu)&=\int_\Omega A(|\nabla tu|)-A_n(|tu|)-\lambda F(|tu|)\,dx\\ &\geq t^{p^+}\int_\Omega A(|\nabla u|)-t^{p_n^+}A_n(|u|)-\lambda t^{r^-} B(|u|)\,dx
\end{align*}

As , $p^+<r^-<p_n^-$ taking limit as $t$ goes to infinity, we can assure that $\F_\lambda(tu)<0$. This completes the proof.
\end{proof}
 
We will denote by $c_\lambda$ the mountain pass level associated to $\F_\lambda$, i.e.
\begin{equation}\label{PS-energy}
c_\lambda := \inf_{\gamma\in\Gamma} \max_{u\in \gamma([0,1])} \F_\lambda(u) \ge 0.
\end{equation}
 
Now, we show that Palais-Smale sequences are bounded.

\begin{lema}\label{acotada1}
Let $\{u_k\}_{k\in\N}\subset W_0^{1,A}(\Omega)$ be a Palais-Smale sequence of $\F_\lambda$, then $\{u_k\}_{k\in\N}$ is bounded in $W_0^{1,A}(\Omega)$.
\end{lema}

\begin{proof}
Let $\{u_k\}_{k\in\N}\subset W^{1,A}_0(\Omega)$ be a Palais-Smale sequence for $\F_\lambda$. Then, by definition
$$
\F_\lambda(u_k)\to c_\lambda \quad  \text{and} \quad  \F_\lambda'(u_k)\to 0.
$$
Now, we have
$$
c_\lambda +1 \geq \F_\lambda(u_k) = \F_\lambda(u_k) - \frac{1}{\gamma} \langle \F_\lambda'(u_k), u_k \rangle + \frac{1}{\gamma} \langle \F_\lambda'(u_k), u_k \rangle,
$$
where
\begin{align*}
\langle \F_\lambda'(u_k), u_k \rangle &= \int_\Omega \frac{a(|\nabla u_k|)\nabla u_k\nabla u_k}{|\nabla u_k|} - \frac{a_n(|u_k|)u_ku_k}{|u_k|} - \lambda f(u_k)u_k\, dx\\
&\leq \int_\Omega p^+ A(|\nabla u_k|) - p_n^-A_n(|u_k|)-\lambda f(u_k)u_k\, dx.\\
\end{align*}
Then, if $p^+<\gamma<p_n^-$ we conclude that
$$
c_\lambda+1 \geq \left(1- \frac{p^+}{\gamma}\right) \int_\Omega A(|\nabla u_k|)\, dx - \frac{1}{\gamma} |\langle \F_\lambda'(u_k), u_k \rangle|.
$$
We can assume that $\|\nabla u_k\|_{A,\Omega}\geq 1$, if not the sequence is bounded. As $\|\mathcal{F}_\lambda'(u_k)\|_{-1,\tilde A}$ is bounded, using Lemma \ref{normayro} we have that
$$
c_\lambda +1 \geq \left(1- \frac{p^+}{\gamma}\right) \|\nabla u_k\|^{p^-}_A - \frac{C}{\gamma}\|\nabla u_k\|_A.
$$
We deduce that $u_k$ is bounded.

This finishes the proof.
\end{proof}

From the fact that $\{u_k\}_{k\in\N}$ is a Palais-Smale sequence of $\F_\lambda$, it follows, from Lemma \ref{acotada1}, that $\{u_k\}_{k\in\N}$ is bounded in $W_0^{1,A}(\Omega)$. Hence passing to a subsequence if necessary, by Theorem \ref{ccp}, we have that
\begin{align}
\label{5.2}&A_n(|u_k|)\rightharpoonup \nu=A_n(|u|) + \sum_{i\in I} \nu_i\delta_{x_i} \quad \nu_i>0,\\
\label{5.3}&A(|\nabla u_k|)\rightharpoonup \mu \geq A(|\nabla u|)+ \sum_{i\in I} \mu_i \delta_{x_i}\quad \mu_i>0,\\
\label{5.4}&S_A \frac{1}{M_n^{-1}\left(\frac{1}{\nu_i}\right)}\le\frac{1}{A_\infty^{-1}\left(\frac{1}{\mu_i}\right)}.
\end{align}
Using these properties of the sequence $\{u_k\}_{k\in\N}$ we derive the following result.

\begin{lema}\label{J is finite}
The set $\left\lbrace x_i\right\rbrace_{i \in I}$  is finite.
\end{lema}

\begin{proof}
Let $x_j$ be fixed and take $\phi \in C_c^{\infty}(\mathbb{R}^n)$ be such that $0 \leq \phi \leq 1$ and
$$
\phi(x)=\begin{cases} 1, & \text{if } |x|\leq 1\\
0, &\text{if } |x| \geq 2.
\end{cases}
$$
For $\varepsilon>0$, define
$$
\phi_\varepsilon(x):=\phi\left(\dfrac{x-x_j}{\varepsilon} \right).
$$

Since $\F'_\lambda(u_k)\to 0$
\begin{align*}
\int_\Omega a(|\nabla u_k|)\frac{\nabla u_k}{|\nabla u_k|}\nabla (\phi_\varepsilon u_k)\,dx&=\int_\Omega a_n(|u_k|)\frac{u_k}{| u_k|}\phi_\varepsilon u_k\,dx+\lambda\int_\Omega f(u_k)u_k\phi_\varepsilon\,dx+o(1)\\
&\leq p_n^+ \int_\Omega A_n(|u_k|)\phi_\varepsilon\,dx+\lambda \gamma\int_\Omega  B(|u_k|)\phi_\varepsilon\,dx+o(1)
\end{align*}

As $B\ll A_n$, we have that $u_k\to u$ in $L^B(\Omega)$, so
$$
\lim_{k\to\infty} \int_\Omega  B(|u_k-u|)\phi_\varepsilon\,dx=0.
$$
By Lemma \ref{Brezis-Lieb}, we obtain that
\begin{align*}
\lim_{k\to\infty} \int_\Omega  B(|u_k|)\phi_\varepsilon\,dx&=\lim_{k\to\infty} \int_\Omega  B(|u_k|)\phi_\varepsilon\,dx-\int_\Omega  B(|u_k-u|)\phi_\varepsilon\,dx\\
&=\int_\Omega  B(|u|)\phi_\varepsilon\,dx.
\end{align*}

On the other hand
\begin{align*}
\int_\Omega a(|\nabla u_k|)\frac{\nabla u_k}{|\nabla u_k|}\nabla (\phi_\varepsilon u_k)\,dx&= \int_\Omega a(|\nabla u_k|)\frac{\nabla u_k}{|\nabla u_k|}\phi_\varepsilon\nabla u_k\,dx+\int_\Omega a(|\nabla u_k|)\frac{\nabla u_k}{|\nabla u_k|}\nabla\phi_\varepsilon u_k\,dx\\ 
&\geq p^-\int_\Omega A(|\nabla u_k|)\phi_\varepsilon\,dx+\int_\Omega a(|\nabla u_k|)\frac{\nabla u_k}{|\nabla u_k|}\nabla\phi_\varepsilon u_k\,dx
\end{align*}
Observe now that $a(|\nabla u_k|)\frac{\nabla u_k\nabla \phi_\ve}{|\nabla u_k|}$ is bounded in $L^{\tilde{A}}(\Omega)$. In fact, using that $\tilde{A}(a(s))=a(s)s - A(s)\leq (p^+-1) A(s)$, we obtain
\begin{align*}
\int_\Omega \tilde{A}\left(a(|\nabla u_k|)\frac{|\nabla u_k\nabla \phi_\ve|}{|\nabla u_k|}\right)\,dx &\leq (p^+-1)\int_\Omega \max\{|\nabla\phi_\ve|^{(p')^+}, |\nabla \phi_\ve|^{(p')^-}\} A(|\nabla u_k|)\,dx\\
&\leq C\int_\Omega A(|\nabla u_k|)\,dx\leq C
\end{align*}
Then $a(|\nabla u_k|)\frac{\nabla u_k}{|\nabla u_k|}\rightharpoonup w_1$ weakly in $(L^{\tilde A}(\Omega))^n$ and $u_k\to u$ in $L^A(\Omega)$. Then
$$
\int_\Omega a(|\nabla u_k|)\frac{\nabla u_k}{|\nabla u_k|}\nabla\phi_\varepsilon u_k\,dx\to\int_\Omega w_1\nabla\phi_\varepsilon u\,dx
$$
Taking limit as $k$ goes to $\infty$
\begin{equation}\label{nu_j-cota}
p^-\int_\Omega\phi_\varepsilon\,d\mu+\int_\Omega w_1\nabla\phi_\varepsilon u\,dx\leq p_n^+\int_{\Omega}\phi_\varepsilon\,d\nu+\lambda \gamma \int_\Omega  B(|u|)\phi_\varepsilon\,dx.
\end{equation}
Now, we want to prove that $\int_\Omega w_1\nabla\phi_\varepsilon u\,dx\to 0$ as $\varepsilon\to0$. In fact, observe first that, for any $v\in W^{1,A}_0(\Omega)$,
$$
0=\lim_{k\to\infty} \langle \F_\lambda(u_k),v \rangle=\lim_{k\to\infty}\left(\int_\Omega a(|\nabla u_k|)\frac{\nabla u_k}{|\nabla u_k|}\nabla v\,dx-\int_\Omega a_n(|u_k|)\frac{u_k}{|u_k|}v+\lambda f(u_k)v\,dx\right).
$$
Now, since  $\{a_n(|u_k|)\frac{u_k}{|u_k|}+\lambda f(u_k)\}_{k\in\N}$ is bounded in $L^{\tilde{A}_n}(\Omega)$,  there exists $w_2$ such that 
$$
a_n(|u_k|)\frac{u_k}{|u_k|}+\lambda f(u_k) \rightharpoonup w_2 \quad\text{ weakly in } L^{\tilde A_n}(\Omega).
$$ 
Then as $\langle \F_\lambda(u_k),v \rangle\to 0$,
$$
\int_\Omega w_1\nabla v-w_2v\,dx=0.
$$
Taking $v=u\phi_{\varepsilon}$
$$
\int_{\Omega} w_1\nabla\phi_\varepsilon u\,dx=-\int_{\Omega} (w_1\nabla u-w_2u)\phi_\varepsilon\,dx
$$
As $ w_1\nabla u-w_2u$ is in $L^1(\Omega)$, the right hand side goes to 0 as $\varepsilon$ goes to 0. Then
$$
\int_\Omega w_1\nabla\phi_\varepsilon u\,dx\to 0,
$$
as we wanted to show.

Moreover $\int_\Omega  B(|u|)\phi_\varepsilon\,dx$ goes to o as $\varepsilon$ goes to 0. Finally, from \eqref{nu_j-cota}, we obtain
$$
p^-\mu_j\leq p_n^+ \nu_j
$$
Then, using \eqref{relacion}, we arrive at 
$$
\inf\{\nu_j,j\in I\}>0.
$$
As $\sum_{j\in I}\nu_j<\infty$, we conclude that $I$ is finite, as we wanted to prove.
\end{proof}

Now we use the following lemma from \cite{FIN}
\begin{lema}[\cite{FIN}, Lemma 4.4]\label{conv en Lp}
For any compact $K \subset \Omega\setminus \bigcup x_j$, there holds that $u_k \to u$ in $L^{A_n}(K)$.
\end{lema}

The next series of lemmas leads us to conclude that the limit $u$ of the Palais-Smale sequence is in fact a weak solution to \eqref{aplicacion}.

\begin{lema}\label{conv gradients} For any compact $K \subset \mathbb{R}^n\setminus \bigcup x_j$, we have
\begin{equation}\label{conclu graidents}
\int_K\left(a(|\nabla u_k|)\frac{\nabla u_k}{|\nabla u_k|}-a(|\nabla u|)\frac{\nabla u}{|\nabla u|}\right) (\nabla u_k-\nabla u)\,dx \to 0,
\end{equation}
as $k \to \infty$.
\end{lema}
  
\begin{proof}
Let $\varphi \in C_c^{\infty}(\mathbb{R}^n)$ be such that $0 \leq \varphi \leq 1$, $\varphi = 1$ in $K$ and $\text{supp}(\varphi)\cap \left\lbrace x_j\right\rbrace_{j\in J}=\emptyset.$

First, observe that, since $A(t)$ is convex, it follows that $a(|\nabla u|) = A'(|\nabla u|)$ is monotone. That is
$$
\left(a(|\nabla u_k|)\frac{\nabla u_k}{|\nabla u_k|}-a(|\nabla u|)\frac{\nabla u}{|\nabla u|}\right)(\nabla u_k-\nabla u)\geq 0.
$$
See, for instance \cite[Section 9.1]{Evans}.

Hence,  we have
\begin{align*}
0&\leq \int_K\left(a(|\nabla u_k|)\frac{\nabla u_k}{|\nabla u_k|}-a(|\nabla u|)\frac{\nabla u}{|\nabla u|}\right) (\nabla u_k-\nabla u)\,dx \\
&\leq \int_\Omega\left(a(|\nabla u_k|)\frac{\nabla u_k}{|\nabla u_k|}-a(|\nabla u|)\frac{\nabla u}{|\nabla u|}\right) (\nabla u_k-\nabla u)\varphi\,dx\\
&=\int_\Omega a(|\nabla u_k|)\frac{\nabla u_k}{|\nabla u_k|}(\nabla u_k-\nabla u)\varphi\,dx-\int_\Omega a(|\nabla u|)\frac{\nabla u}{|\nabla u|} (\nabla u_k-\nabla u)\varphi\,dx
\end{align*}

First observe that $a(|\nabla u|)\frac{\nabla u}{|\nabla u|}\varphi\in L^{\tilde{A}}$ and, since $\nabla u_k\rightharpoonup \nabla u$ weakly in $L^A(\Omega)$,
$$
\int_\Omega a(|\nabla u|)\frac{\nabla u}{|\nabla u|} (\nabla u_k-\nabla u)\varphi \,dx\to 0.
$$

As $\{(u_k-u)\varphi\}_{k\in\N}$ is bounded in $W^{1,A}_0(\Omega)$, then $\langle\F_\lambda(u_k),(u_k-u)\varphi\rangle\to 0$. Therefore
\begin{align*}
\langle\F_\lambda(u_k),(u_k-u)\varphi\rangle &=\int_\Omega a(|\nabla u_k|)\frac{\nabla u_k}{|\nabla u_k|}\nabla[(u_k-u)\varphi]\,dx-\int_\Omega a_n(|\nabla u_k|)\frac{\nabla u_k}{|\nabla u_k|} (u_k-u)\varphi\,dx\\&-\int_\Omega\lambda f(u_k)(u_k-u)\varphi\,dx
\end{align*}
Using that $ a_n(|\nabla u_k|)\frac{\nabla u_k}{|\nabla u_k|}+\lambda f(u_k)$ is bounded in $L^{\tilde{A}_n}(\Omega)$ and, from Lemma \ref{conv en Lp} $u_k\to u$ in $L^{A_n}(\supp(\varphi))$, we have
$$
\int_\Omega \left(a_n(|\nabla u_k|)\frac{\nabla u_k}{|\nabla u_k|}+\lambda f(u_k)\right) (u_k-u) \varphi\,dx\to 0
$$
Then
$$
\int_\Omega a(|\nabla u_k|)\frac{\nabla u_k}{|\nabla u_k|}\nabla[( u_k- u)\varphi]\,dx\to 0.
$$
But,
$$
\left|\int_\Omega a(|\nabla u_k|)\frac{\nabla u_k}{|\nabla u_k|}\nabla\varphi (u_k-u)\,dx\right|\leq \left\|a(|\nabla u_k|)\frac{\nabla u_k}{|\nabla u_k|}\right\|_{\tilde{A}}\|u_k-u\|_A \|\nabla\varphi\|_\infty\to 0,
$$
so
$$
\int_\Omega a(|\nabla u_k|)\frac{\nabla u_k}{|\nabla u_k|}(\nabla u_k- \nabla u)\varphi\,dx\to 0,
$$
as we wanted to show.
\end{proof}

Now, arguing as in \cite{FIN} we obtain
\begin{corol}[\cite{FIN}, Corollary 4.6] \label{puntual}
Let $\{u_k\}_{k\in\N}$ be a Palais-Smale sequence of $\F_\lambda$.  There exists a subsequence $\{u_{k_j}\}_{j\in\N}\subset \{u_k\}_{k\in\N}$ such that $\nabla u_{k_j}\to\nabla u$ a.e in $\Omega$.
\end{corol}

All of the above lemmas, can be put into a single statement.
\begin{lema}\label{convergencia}
Given $\{u_k\}_{k\in\N}\subset W^{1,A}_0(\Omega)$ a Palais-Smale sequence of $\F_\lambda$, there exists a subsequence $\{u_{k_j}\}_{j\in\N}\subset \{u_k\}_{k\in\N}$ such that
\begin{align}\label{limita}
\lim_{j\to\infty} \int_\Omega a(|\nabla u_{k_j}|)\frac{\nabla u_{k_j}}{|\nabla u_{k_j}|}\nabla\varphi\,dx &=\int_\Omega a(|\nabla u|)\frac{\nabla u}{|\nabla u|}\nabla\varphi\,dx\\
\label{limitan}
\lim_{j\to\infty}  \int_\Omega a_n(|u_{k_j}|)\frac{u_{k_j}}{|u_{k_j}|}\varphi\,dx &= \int_\Omega a_n(|u|)\frac{u}{|u|}\varphi\,dx  \\
\label{limitf}
\lim_{j\to\infty}  \int_\Omega f(u_{k_j})\varphi\,dx &=  \int_\Omega f(u)\varphi \,dx
\end{align}
for any $\varphi\in C^\infty_c(\Omega)$.
\end{lema}

\begin{proof}
We start by proving \eqref{limita}. As $a(|\nabla u_k|)\frac{\nabla u_k}{|\nabla u_k|}$ is bounded in $L^{\tilde{A}}(\Omega)$,  then there exists $w\in \left(L^{\tilde A}(\Omega)\right)^n$ such that, up to a subsequence, $a(|\nabla u_k|)\frac{\nabla u_k}{|\nabla u_k|}\rightharpoonup w$. Moreover, by Corollary \ref{puntual}, passing eventually to a further subsequence,
$$
a(|\nabla u_k|)\frac{\nabla u_k}{|\nabla u_k|}\to a(|\nabla u|)\frac{\nabla u}{|\nabla u|}\qquad\mbox{a.e in }\Omega.
$$ 
Then, 
$$
a(|\nabla u_k|)\frac{\nabla u_k}{|\nabla u_k|}\rightharpoonup a(|\nabla u|)\frac{\nabla u}{|\nabla u|},\quad \text{weakly in } \left(L^{\tilde A}(\Omega)\right)^n,
$$
as we wanted to prove. The proof of \eqref{limitan} and \eqref{limitf}, follows in much the same way.
\end{proof}
  
\begin{corol}\label{u.sol}
If $u$ is the limit of a Palais-Smale sequence, then $u$ is a weak solution to \eqref{aplicacion}.
\end{corol}

\begin{proof}
Let $\{u_k\}$ be a Palais-Smale sequence and $\varphi\in C^\infty(\Omega)$, then
$$
\langle \F'_\lambda(u_k), \varphi\rangle=\int_\Omega a(|\nabla u_k|)\frac{\nabla u_k}{|\nabla u_k|}\varphi\,dx-\int_\Omega a_n(|u_k|)\frac{u_k}{|u_k|}\varphi\,dx- \lambda\int_\Omega f(u_k)\varphi\,dx.
$$
Taking limit, by Lemma \ref{convergencia}, we obtain that
$$
\int_\Omega a(|\nabla u|)\frac{\nabla u}{|\nabla u|}\varphi\,dx-\int_\Omega a_n(|u|)\frac{u}{|u|}\varphi\,dx- \lambda\int_\Omega f(u)\varphi\,dx=0
$$
As we wanted to proof.
\end{proof}

In order to finish the proof of Theorem \ref{r>p}, it remains to see that the function $u$ that is the limit of a Palais-Smale sequence is nontrivial.

For that purpose, we need an asymptotic behavior of the Mountain Pass energies as $\lambda\to \infty$.
\begin{lema}\label{cl0}
Let $c_\lambda$ be the mountain pass energy level of $\F_\lambda$ given by \eqref{PS-energy}. Then
$$
\lim_{\lambda\to\infty} c_\lambda = 0.
$$
\end{lema}  

\begin{proof}
Let $u_0\in C^\infty_c(\Omega)$ be such that $u_0\ge 0$ and $u_0\neq 0$.
Observe that
\begin{equation}\label{cl1}
0\le c_\lambda \le \max_{t>0} \F_\lambda(t u_0).
\end{equation}
Now, If $t\ge 1$,
$$
\F_\lambda(t u_0)\le t^{p^+} \int_\Omega A(|\nabla u_0|)\, dx - t^{p_n^-} \int_\Omega A_n(|u_0|)\, dx.
$$
Hence, there exists $T_0\ge 1$, independent of $\lambda$, such that $\F_\lambda(t u_0)$ attains a maximum in some $t_\lambda\in (0,T_0)$.

We claim that $t_\lambda\to 0$ as $\lambda\to\infty$. In fact, if not, there exists $\lambda_j\to\infty$ and $\delta>0$ such that
$$
0<\delta\le t_{\lambda_j}\le T_0.
$$
But then
$$
\F_{\lambda_j}(t_{\lambda_j} u_0)\le T_0^{p^+} \int_\Omega A(|\nabla u_0|)\, dx - \lambda_j \min_{t\in [\delta, T_0]}\int_\Omega F(t|u_0|)\, dx\to -\infty
$$
as $j\to\infty$, a contradiction. Hence, $t_\lambda\to 0$, and using \eqref{cl1} we get
$$
0\le c_\lambda\le \F_\lambda(t_\lambda u_0)\le t_\lambda^{p^-} \int_\Omega A(|\nabla u_0|)\, dx - t_\lambda^{p_n^+} \int_\Omega A_n(|u_0|)\, dx \to 0
$$
as $\lambda\to\infty$.
\end{proof}

Next, we show a lemma that says that if the energy level is small, then the limit of any Palais-Smale sequence is nontrivial.

\begin{lema}\label{cl<c}
Let $\{u_k\}_{k\in\N}$ be a Palais-Smale sequence for $\F_\lambda$. Then, there exists $c>0$ such that if $c_\lambda<c$ and $u_k\rightharpoonup u$ weakly in $W^{1,A}_0(\Omega)$, then $u\neq 0$.
\end{lema}

\begin{proof}
Assume by contradiction that $u=0$. Now since $\{u_k\}_{k\in\N}$ is bounded in $W^{1,A}_0(\Omega)$ it follows that each of the integrals in
$$
\langle \F'_\lambda(u_k), u_k\rangle = \int_\Omega a(|\nabla u_k|)|\nabla u_k|\, dx - \int_\Omega a_n(|u_k|)|u_k|\, dx - \lambda \int_\Omega f(u_k)u_k\, dx 
$$
is bounded. 

Also, observe that 
$$
\langle \F'_\lambda(u_k), u_k\rangle\to 0 \quad\text{and}\quad \int_\Omega |f(u_k)| |u_k|\, dx \le \gamma \int_\Omega |F(u_k)|\, dx \le \gamma \int_\Omega B(|u_k|)\, dx \to 0
$$
as $k\to\infty$ (for the second term we are using our assumption that $u=0$).

Hence, passing to a subsequence if necessary, we get that
$$
\lim_{k\to\infty} \int_\Omega a(|\nabla u_k|)|\nabla u_k|\, dx = \lim_{k\to\infty} \int_\Omega a_n(|u_k|)|u_k|\, dx =: m
$$
Next, observe that
$$
\F_\lambda(u_k) \le \frac{1}{p^-}\int_\Omega a(|\nabla u_k|)|\nabla u_k|\, dx - \lambda \int_\Omega F(|u_k|)\, dx,
$$
hence
$$
0<c_\lambda  =\lim_{k\to\infty} \F_\lambda(u_k)\le \frac{m}{p^-},
$$
from where it follows that $m>0$. Now we want to give a lower bound for $m$ independent of $\lambda$. In fact,
\begin{align*}
\|\nabla u_k\|_A &\le \max\left\{\left(\int_\Omega A(|\nabla u_k|)\, dx\right)^{1/p^+}, \left(\int_\Omega A(|\nabla u_k|)\, dx\right)^{1/p^-}\right\}\\
&\le \max\left\{\left(\frac{1}{p^-}\int_\Omega a(|\nabla u_k|)|\nabla u_k|\, dx\right)^{1/p^+}, \left(\frac{1}{p^-}\int_\Omega a(|\nabla u_k|)|\nabla u_k|\, dx\right)^{1/p^-}\right\}.
\end{align*}
Hence
\begin{equation}\label{m1}
\limsup_{k\to\infty} \|\nabla u_k\|_A \le \max\left\{\left(\frac{m}{p^-}\right)^{1/p^+},\left(\frac{m}{p^-}\right)^{1/p^-}\right\}. 
\end{equation}
Now, we argue analogously,
\begin{align*}
\|u_k\|_{A_n} &\ge \min\left\{\left(\int_\Omega A_n(|u_k|)\, dx\right)^{1/p_n^+},\left( \int_\Omega A_n(|u_k|)\, dx\right)^{1/p_n^-}\right\}\\ 
&\ge \min\left\{\left(\frac{1}{p_n^+}\int_\Omega a_n(|u_k|)|u_k|\, dx\right)^{1/p_n^+}, \left(\frac{1}{p_n^+}\int_\Omega a_n(|u_k|)|u_k|\, dx\right)^{1/p_n^-}\right\}.
\end{align*}
Hence
\begin{equation}\label{m2}
\liminf_{k\to\infty} \|u_k\|_{A_n} \ge \min\left\{\left(\frac{m}{p_n^+}\right)^{1/p_n^+}, \left(\frac{m}{p_n^+}\right)^{1/p_n^-}\right\}.
\end{equation}

So, using \eqref{m1}, \eqref{m2} together with the Sobolev inequality \eqref{SA.def}, we arrive at
$$
S_A\min\left\{\left(\frac{m}{p_n^+}\right)^{1/p_n^+}, \left(\frac{m}{p_n^+}\right)^{1/p_n^-}\right\}\le \max\left\{\left(\frac{m}{p^-}\right)^{1/p^+},\left(\frac{m}{p^-}\right)^{1/p^-}\right\}.
$$
Therefore, since $m>0$, there exists a constant $M>0$ depending on $S_A$, $p^\pm$ and $p_n^\pm$ such that $m\ge M$.

Finally, observe that
$$
\F_\lambda(u_k) \ge \frac{1}{p^+}\int_\Omega a(|\nabla u_k|)|\nabla u_k|\, dx - \frac{1}{p_n^-} \int_\Omega a_n(|u_k|)|u_k|\, dx -\lambda \int_\Omega F(|u_k|)\, dx.
$$
Passing to the limit as $k\to\infty$, we obtain that
$$
c_\lambda\ge \left(\frac{1}{p^+}-\frac{1}{p_n^-}\right) m \ge \left(\frac{1}{p^+}-\frac{1}{p_n^-}\right) M.
$$
So, denoting $c=\left(\frac{1}{p^+}-\frac{1}{p_n^-}\right) M$, if $c_\lambda<c$ we arrive at a contradiction, and hence $u\neq0$.
\end{proof}

Now we are in position to finish with the proof of Theorem \ref{r>p}.

\begin{proof}[Proof of Theorem \ref{r>p}]
The proof is now just a combination of Lemma \ref{I satisfies MPL}, Corollary \ref{u.sol} and Lemmas \ref{cl0} and \ref{cl<c}.
\end{proof}

\section*{Acknowledgements}
A.S. is partially supported by ANPCyT under grants PICT 2017-0704, PICT 2019-3837 and by Universidad Nacional de San Luis under grants PROIPRO 03-2420.

J.F.B. is partially supported by UBACYT Prog. 2018 20020170100445BA and by ANPCyT PICT 2019-00985. 

Both authors are members of CONICET and are grateful for the support.

\bibliography{biblio}
\bibliographystyle{plain}

\end{document}